\documentclass{amsart}
\usepackage[nobysame]{amsrefs}
\usepackage{amsfonts}
\usepackage{graphicx}
\usepackage{amsmath}
\usepackage{amssymb}
\usepackage{latexsym}
\usepackage{url}
\usepackage{verbatim}

\newtheorem{theorem}{Theorem}[section] 
\newtheorem{lemma}[theorem]{Lemma}     
\newtheorem{corollary}[theorem]{Corollary}
\newtheorem{proposition}[theorem]{Proposition}

\newcommand{\newnumbered}[2]{\newtheorem{#1}{#2}} 

\newnumbered{assertion}{Assertion}    
\newnumbered{conjecture}{Conjecture}  
\newnumbered{definition}{Definition}
\newnumbered{hypothesis}{Hypothesis}
\newnumbered{remark}{Remark}
\newnumbered{note}{Note}
\newnumbered{observation}{Observation}
\newnumbered{problem}{Problem}
\newnumbered{question}{Question}
\newnumbered{algorithm}{Algorithm}
\newnumbered{example}{Example}

\newnumbered{heuristic}{Heuristic}
\newnumbered{prediction}{Prediction}

\newcommand{\set}[1]{\left\{#1\right\}}
\newcommand{\abs}[1]{\left\vert#1\right\vert}
\newcommand{\sigzero}{\tau}
\newcommand{\st}{\colon}
\newcommand{\ord}{{\mathop{\mathrm{ord}}\nolimits}}

\title[{Fixed points of the self-power map}]
{Statistics for fixed points of the self-power
  map} 

\author{Matthew Friedrichsen}
\address{ Madison, WI\\
 USA}
    \email{friedrichsenm@gmail.com}
\author{ Joshua Holden}
\address{ Department of Mathematics\\
 Rose-Hulman Institute of Technology\\
 5500 Wabash Avenue\\
 Terre Haute, IN 47803-3999 \\
 USA}
    \email{holden@rose-hulman.edu}


\begin{document}
\maketitle

\begin{abstract}
  The map $x \mapsto x^x$ modulo $p$ is related to a variation of the
  ElGamal digital signature scheme in a similar way to the discrete
  exponentiation map, but it has received much less study.  We explore
  the number of fixed points of this map by a statistical analysis of experimental data.  In
  particular, the number of fixed points can in many cases be modeled
  by a binomial distribution.  We discuss the many cases where this
  has been successful, and also the cases where a good model may not
  yet have been found.


\end{abstract}

\section{Introduction and Motivation}

The security of the ElGamal digital signature scheme against selective
forgery relies on the
difficulty of solving the congruence
$g^{H(m)} \equiv y^r r^s \pmod{p}$
for $r$ and $s$, given $m$, $g$, $y$, and $p$ but not knowing the
discrete logarithm of $y$ modulo $p$ to the base $g$.  (We assume
for the moment the security of the hash function $H(m)$.)  
Similarly,
the security of a certain variation of this scheme given in, e.g.,
\cite[Note~11.71]{handbook}, relies on the difficulty of solving
\begin{equation} \label{eq:variation}
g^{H(m)} \equiv y^s r^r \pmod{p} .
\end{equation}
It is generally expected that the best way to solve either of these
congruences is to calculate the discrete logarithm of $y$, but this is
not known to be true.  In particular, another possible option would be
to choose $s$ arbitrarily and solve the relevant equation for $r$.  In
the case of~\eqref{eq:variation}, this boils down to solving equations
of the form $x^x \equiv c \pmod{p}$.  We will refer to these equations
as ``self-power equations'', and we will call the map $x \mapsto x^x$
modulo $p$ the ``self-power map''.  This map has been studied in
various forms in \cite{anghel, balog_et_al, crocker66, crocker69,
  somer, holden02, holden02a, holden_moree, REU2010, holden_robinson,
  kurlberg_et_al}.  In this work we will investigate experimentally
the number of fixed points of the map, i.e., solutions to
\begin{equation} \label{eq:spfp}
x^x \equiv x \pmod{p}
\end{equation}
between $1$ and $p-1$.  In particular, we would like to know whether
the distribution across various primes behaves as we would expect if
the self-power map were a ``random map''.  We do this by creating a
model in which values of a map are assumed to occur uniformly randomly
except as forced by the structure of the self-power map.  We can then
predict the distribution of the number of fixed points of this random
map and compare it statistically to the actual self-power map.  If
there is ``nonrandom'' structure in the self-power map, it may be
possible to exploit that structure to break the signature scheme
mentioned above or others like it.

Some theoretical work has been done on bounding the possible number of
fixed points of the self-power map.  If we denote the number of
solutions to \eqref{eq:spfp} which fall between $1$ and $p-1$ by
$F(p)$, then we have:

\begin{theorem}[See Section~5 of~\cite{balog_et_al} and Section~1.1
  of~\cite{kurlberg_et_al}] \label{thm:antsx1}
$F(p) \leq p^{1/3+o(1)}$ as $p \to \infty$.
\end{theorem}

As far as a lower bound, every $p$ has at least $x=1$ as a solution to
\eqref{eq:spfp}, and at least some primes have only this solution.
However, while~\cite{kurlberg_et_al} gives good reason to believe that
there are infinitely many such primes, they also prove that these primes
are fairly rare:

\begin{theorem}[Theorem~1 of~\cite{kurlberg_et_al}]  
  Let $\pi(N)$ be the number of primes less than or equal to $N$ as
  usual.  Let $\mathcal{A}(N)$ denote the set of primes less than or
  equal to $N$ such that $F(p)=1$.  Then
$$\#\mathcal{A}(N) \leq \frac{\pi(N)}{(\ln \ln \ln N)^{0.4232+o(1)}}$$
  as $N \to \infty$.
\end{theorem}

\section{Models and Experimental Results}

\subsection{Heuristics and Normality}

Theorem~\ref{thm:antsx1} gives us a range in which the number of fixed
points $F(p)$ can lie, but does not say anything about the
distribution of the values within that range.  As described above, our
goal is to create a random model for the self-power map much like was
done for the discrete exponential map in \cite{holden02, holden02a,
  holden_moree}.  Our first attempt assumed that $F(p)$ was normally
distributed around the predicted value $\sum_{d \mid
  p-1}\frac{\phi(d)}{d}$.  (The normality assumption had been
successfully used for the discrete exponential map in, e.g.,
\cite{cloutier_holden}, see also~\cite{sorenson}.  Furthermore, it
appeared to be justified by the Central Limit Theorem given the number
of primes we were intending to test.)

In order to calculate the variance of $F(p)$, we use the
following heuristic, which is related to those
in~\cite[Section~6]{holden_moree}, and can also be derived from the
assumptions in~\cite[Section~4.1]{kurlberg_et_al}.   




\begin{heuristic} \label{xtoxxheur}
    The map $x \mapsto x^{x} \bmod{p}$ is a random map 
     in the sense that
    for all $p$, if $x, y$ are chosen uniformly at random from
    $\set{1, \ldots, p-1}$ with $\ord_p x = d$, then 
$$\Pr[x^x \equiv y \pmod{p}] \approx 
\begin{cases} \frac{1}{d} & \text{if\ }  \ord_p y \mid d,\\
0 & \text{otherwise.}
\end{cases}$$
\end{heuristic}

As some justification, one can use the methods of
\cite[Cor.~6.2]{holden_robinson} to prove the following lemma.  This
shows that the heuristic holds exactly  over the
range $1 \leq x \leq (p-1)p$ rather than $1 \leq x \leq p-1$:

\begin{lemma} \label{xtoxxlemma} For all $p$, given fixed $d \mid (p-1)$ and
  fixed $y \in \set{1, \ldots, (p-1)p}$, $p \nmid y$, such that $\ord_p y
  \mid d$, then
$$      \#\set{x \in \set{1, \ldots, (p-1)p} \st p \nmid x,\  x^{x} \equiv
        y \pmod{p},\ \ord_p x = d }
      = \frac{p-1}{d} .$$
  \end{lemma}
  
From there one can prove the following theorem:

\begin{theorem} \label{thm:antsx2}  Let $G(p)$ be 
the number of solutions to  \eqref{eq:spfp} with $1 \leq x
\leq (p-1)p$ and $p \nmid x$.   Then
\begin{equation*}
G(p)  = (p-1) \sum_{n \mid  p-1}\frac{\phi(n)}{n}
\end{equation*}
\end{theorem}

\begin{proof} Sum up the numbers in Lemma~\ref{xtoxxlemma} given $y=x$
  over the possible values of $d$.  (Note that this could also be
  proved using the method of~\cite[Theorem~1]{somer}.)
\end{proof}

As far as using Heuristic~\ref{xtoxxheur}, note that it implies that
the ``experiment'' of testing whether $x$ is a fixed point behaves as
a Bernoulli trial.  Let $F_d(p)$ be the number of solutions
to~\eqref{eq:spfp} with $1 \leq x \leq p-1$ and $\ord_p x = d$.
Assuming independence of the Bernoulli trials (which is
not completely accurate, as we shall see), $F_d(p)$ is 
distributed as a binomial random variable with $\phi(d) = \#\set{x \in
  \set{1, \ldots, p-1} \st \ord_p x = d}$ trials and success
probability ${1}/{d}$. This distribution has mean
${\phi(d)}/{d}$, as expected, and variance
${{\phi(d)(d-1)}/{d^2}}$. Summing over $d \mid p-1$ gives the
predicted mean and variance of $F(p)$.

We tested the hypothesis that $F(p)$ was normal with this mean and
variance by collecting data for 238 primes from 100,003 to
102,667 and 599 primes from 1,000,003 to 1,007,977.  The number of fixed
points for each prime was determined using C code originally written
by Cloutier \cite{cloutier_holden} and modified by Lindle
\cite{lindle}, Hoffman \cite{hoffman}, and
Friedrichsen-Larson-McDowell \cite{REU2010} Post-processing was done
using a Python script written by the first author.  The code was run
on servers maintained by the Rose-Hulman Computer Science-Software
Engineering and Mathematics Departments and took only a few minutes of
computational time.

Once the values of $F(p)$ were collected, they were normalized to a
$z$-statistic by subtracting the predicted mean and dividing by the
predicted standard deviation (square root of the variance).  The
$z$-statistics were grouped separately for the six-digit and
seven-digit primes and tested to see if they conformed to the expected
standard normal distribution.  As you can see in
Figures~\ref{fig:SPP6hist} and~\ref{fig:SPP7hist}, the distributions
appear to be roughly normal to the naked eye, and the standard
deviations are close to 1 as expected.  The means are closer to 0.5
than the expected 0, and there are a few bars which seem significantly
off, but these features could be attributed to certain known
properties which appear below in Theorem~\ref{thm:predorders}.  More
troubling is the lack of normality revealed by probability plots in
Figures~\ref{fig:SPP6prob} and~\ref{fig:SPP7prob}.  Perfectly normal
distributions would lie along the diagonal lines in these figures, and
Ryan-Joiner tests confirm that it is very unlikely that $F(p)$ is
obeying a normal distribution for these primes.  In fact there appear
to be more primes in the ``tails'' than expected, that is, a larger
than expected number of primes with significantly more or fewer fixed
points than expected.

\begin{figure}
  \centering
  \includegraphics[height=2in]{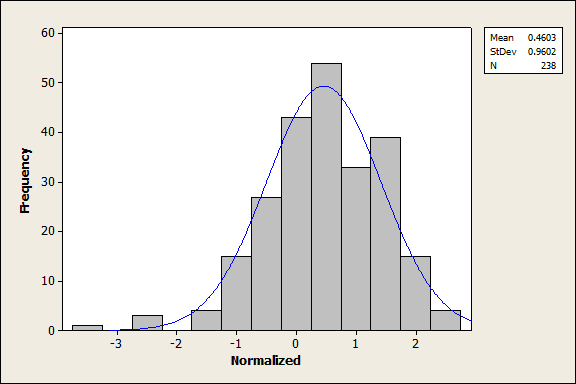}
  \caption{Histogram of $z$-statistics for six-digit primes}
  \label{fig:SPP6hist}
\end{figure}

\begin{figure}
  \centering
  \includegraphics[height=2in]{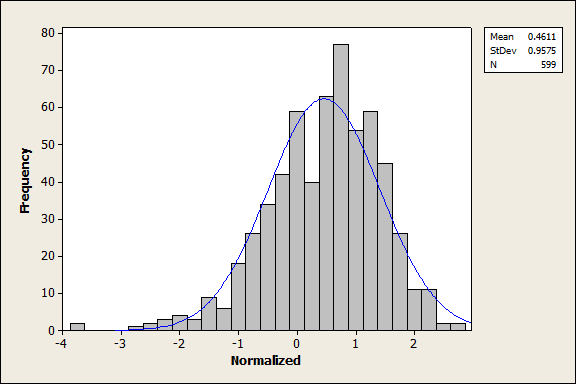}
  \caption{Histogram of $z$-statistics for seven-digit primes}
  \label{fig:SPP7hist}
\end{figure}

\begin{figure}
  \centering
  \includegraphics[height=2in]{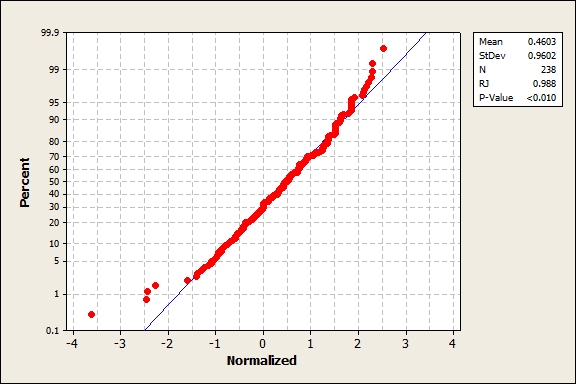}
  \caption{Probability Plot of $z$-statistics for six-digit primes}
  \label{fig:SPP6prob}
\end{figure}

\begin{figure}
  \centering
  \includegraphics[height=2in]{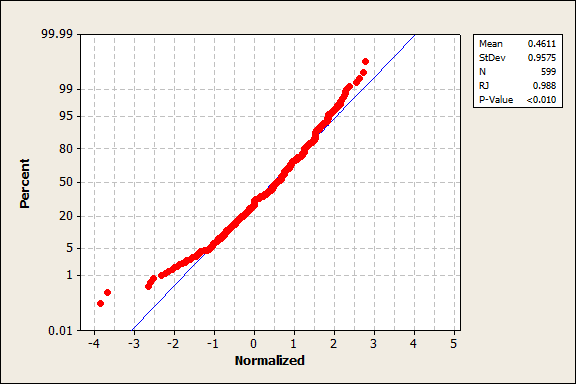}
  \caption{Probability Plot of $z$-statistics for seven-digit primes}
  \label{fig:SPP7prob}
\end{figure}

\subsection{Binomial Distribution and Goodness of Fit}

Some modification of the code by the first author allowed us to
collect the values of $F_d(p)$ for the same primes as above, in order
to see if particular orders were behaving less ``randomly'' than
others.  We excluded certain orders where $F_d(p)$ is  known to behave
predictably:

\begin{theorem} \label{thm:predorders} \mbox{}
\begin{enumerate}
\item $F_1(p)=1$ for all $p$. \label{F1}
\item $F_2(p)=0$ for all $p$.\label{F2}
\item $F_{p-1}(p)=0$ for all $p$. \label{Fp-1}
\item $F_{(p-1)/2}(p)=
\begin{cases}
0 & \text{if~} p \equiv 3 \text{~or~} 5 \pmod{8}, \text{or if~} p \equiv 1 \text{~or~} 7 \pmod{8} \text{~and~}
\ord_p 2  \neq  (p-1)/2; \\
1 & \text{if~} p \equiv 1 \text{~or~} 7 \pmod{8} \text{~and~}
\ord_p 2 = (p-1)/2.
\end{cases}$
\label{Fp-1/2}
\end{enumerate}
\end{theorem}

To prove this we use the following lemmas:

\begin{lemma}[Proposition~7 of~\cite{REU2010}] \label{prop7}
Let $p$ be prime.  The number $x$ is a solution to~\eqref{eq:spfp} if and
only if $x \equiv 1 \pmod{\ord_p x}$.
\end{lemma}

\begin{corollary} \label{prop7cor}
  Let $d \mid (p-1)$.  The solutions to~\eqref{eq:spfp} of order $d$
  are exactly the elements of $\mathcal{P} = \set{1, d+1, 2d+1,
    \ldots, p-d}$ which have order $d$.
\end{corollary}

\begin{proof}[Proof of Theorem~\ref{thm:predorders}]
  Parts~\ref{F1} and~\ref{F2} are clear from the definition.
  Part~\ref{Fp-1} is Proposition~6 of~\cite{REU2010}.  If $x$ is a
  fixed point such that $\ord_p x = (p-1)/2$, then
  Corollary~\ref{prop7cor} implies that $x=(p+1)/2$.  Then
  Proposition~2 of~\cite{REU2010} tells us $x$ is a fixed point if and
  only if $2$ is a quadratic residue modulo $p$, which is if and only
  if $p \equiv 1$ or $7 \pmod{8}$.  Combining this with the fact that
  $\ord_p (p+1)/2 = \ord_p 2$ gives Part~\ref{Fp-1/2}.
\end{proof}

\begin{remark}
Note that the behavior of fixed points in safe primes, where
$(p-1)/2$ is prime, is completely explained by Theorem~\ref{thm:predorders}.
\end{remark}

We collected values of $F_d(p)$ for each prime and each value of $d
\mid p-1$ other than $d=1,$ $2,$ $p-1,$ and $(p-1)/2$.  We then attempted
to normalize this data, but the resulting $z$-statistics turned out to be
too highly clustered and did not resemble normal data.  We therefore
decided to do a chi-squared goodness-of-fit test on the data.  We used
the formula for the mass function of a binomial distribution to
predict that 

\begin{prediction} \label{pred1}
$\Pr[F_d(p)= k] = \binom{\phi(d)}{k} \left(\frac{1}{d}\right)^k
\left(\frac{d-1}{d}\right)^{\phi(d)-k}$
\end{prediction}

We chose to use the categories $k=0$, $k=1$, $k=2$, and $k>2$ for our
test in order to make sure the categories with large $k$ did not get
too small.  We summed the predictions over $p$ and $d$ for each of the
categories and compared them with the observed numbers of $p$ and $d$
which fell into each category.  The resulting chi-squared statistic
was $4.66$, giving a $p$-value of $0.198$.  Using the common cutoff of
$p=0.05$ for statistical significance, we do not see statistical
evidence that our predictions are incorrect.

However, not all values of $p$ and $d$ fit the predictions equally
well.  We tested this by sorting in various ways the values of
$F_d(p)$ collected for $p$ between 100,003 and 102,667, and $d \mid
p-1$ other than $d=1,$ $2,$ $p-1,$ and $(p-1)/2$.  After each sort, we
calculated the chi-squared statistics and $p$-values for a sliding
window of 100 values, with predictions and observations calculated as
above.  (The size of the window was chosen in order to make sure there
were enough data points in the window for the chi-squared test to be
valid.) 

The strongest evidence of a pattern was seen when the data was
sorted by value of $d$, as can be seen in
Figure~\ref{fig:chisq6-sliding}.  For data randomly generated
according to the relevant binomial distributions, $p$-values should be
evenly distributed between $0$ and $1$.  When $p$-values are biased
towards 0 it indicates statistically significant divergence from the
predicted distributions.  In other words, dots on the same
(approximate) horizontal line should be evenly distributed between the
left- and right-hand sides of the graph.  (Note that the value of $d$
used to place the dot on the plot is the largest value of $d$ in the
window of 100 pairs, so some dots would more accurately ``belong'' to
more than one line.)  Horizontal lines where the dots are clustered
towards the left-hand side indicate statistically significant
divergence.

\begin{figure}
  \centering
  \includegraphics[height=3in]{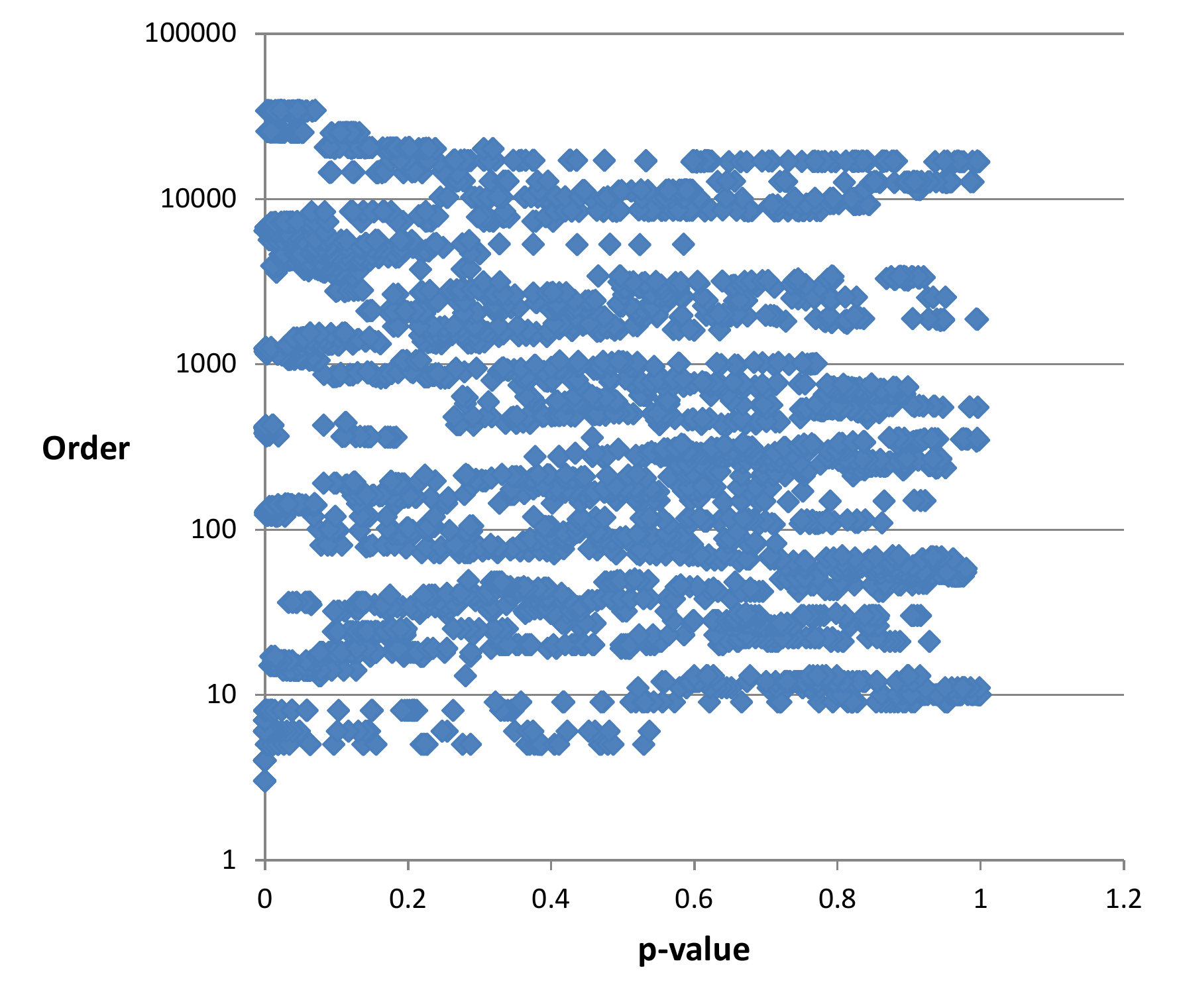}
  \caption{Logarithmic plot showing $p$-values of the sliding window
    goodness-of-fit test, data sorted by order, for six-digit primes}
  \label{fig:chisq6-sliding}
\end{figure}

As you can see, the strongest divergence from the predictions occurs
with particularly small and particularly large values of $d$.  (Since
the value of $d$ used to place the dot on the plot is the largest
value in the window, the effect for small $d$ is even larger
than it appears in the plot.)  We therefore looked for theoretical
explanations of these effects.

\section{Small and Large Orders}

\subsection{Small Orders}

For $d=3$ we observed that while $F_3(p)=2$ should occur roughly
one-ninth of the time according to Prediction~\ref{pred1}, it never
occurred at all in our data.  A similar but less striking effect was
observed for $d=4$, while for $d=6$ it was $F_6(p)=1$ which was never
observed, despite Prediction~\ref{pred1} saying it should happen over
one-quarter of the time.  It turns out that there is a significant
lack of independence in the fixed points for these orders, as we were
able to show.

\begin{theorem}
\begin{enumerate}
\item $F_3(p)=0$ or $F_3(p)=1$ for all $p$ such that $3 \mid (p-1)$. \label{F3}
\item $F_4(p)=0$ or $F_4(p)=1$ for all $p$ such that $4 \mid (p-1)$. \label{F4}
\item $F_6(p)=0$ or $F_6(p)=2$ for all $p$ such that $6 \mid (p-1)$. \label{F6}
\end{enumerate}
\end{theorem}

\begin{proof} 
  If $3 \mid (p-1)$, then by Lemma~\ref{prop7} the fixed points of
  order 3 are exactly the elements congruent to 1 modulo 3.  In this
  case there are two elements of order 3, and a direct computation
  shows that if $x$ is one of them, then $p-1-x$ is the other.  Thus
  the elements of order 3 add up to $p-1 \equiv 0 \pmod{3}$. So at
  most one of the elements of order 3 can be a fixed point, proving
  Part~\ref{F3}.  Part~\ref{F4} is similar except that the elements of
  order 4 add up to $p \equiv 1 \pmod{4}$.  In Part~\ref{F6} the
  elements of order 6 add up to $p+1 \equiv 2 \pmod{6}$ so if one is a
  fixed point then the other must be also.
\end{proof}

The following lemma says that the elements of a given order $f$ are
approximately uniformly distributed across the residue classes modulo
any given $r$.

\begin{lemma} \label{cz-eq7}
  Let $a$, $r$, and $f$ be positive integers such that $0 \leq a < r
  \leq p-1$  and \mbox{$f \mid (p-1)$}.  Let $\mathcal{Q}
  = \set{a, r+a, 2r+a, \ldots, p-1-r+a}$. Let $\mathcal{Q}' = \set{ x \in
    \mathcal{Q} \st \ord_{p}(x)=f}$.  Then
\[
\abs{\#\mathcal{Q}' - \frac{\phi(f)}{r}} \leq
\sigzero(p-1)\sqrt{p}(1+\ln p)
\]
where $\sigzero(p-1)$ is the number of divisors of $p-1$.
\end{lemma}

\begin{proof} The proof is the same as the proof of Equation~(7)
  from~\cite{cobeli-zaharescu} with the order
  equal to $f$ instead of $p-1$.
\end{proof}

In particular, we would expect the elements of order $d$ to be equally
likely to be of any residue class modulo $d$.  This leads us to
predict that:

\begin{prediction}
\begin{enumerate}
\item $\Pr[F_3(p)=0] = 1/3$ and $\Pr[F_3(p)=1] = 2/3$
\item $\Pr[F_4(p)=0] = 1/2$ and $\Pr[F_4(p)=1] = 1/2$
\item $\Pr[F_6(p)=0] = 5/6$ and $\Pr[F_6(p)=2] = 1/6$
\end{enumerate}
\end{prediction}

This is in fact what we observe in the data, as shown in
Figure~\ref{fig:smallorder6}.  This figure shows the number of primes
such that $d \mid (p-1)$ for $d=3, 4,$ and $6$, the number of primes
for each $d$ with $F_d(p)=0, 1$, and $2$, and the $p$-value given by a
chi-squared test against the distribution predicted above.  Once
again, we do not see statistical evidence that our predictions are
incorrect.  (Other small orders do not seem to exhibit this lack of
independence in a statistically significant way.  For example, $d=5$
fits the distribution of the original model with $p=0.222$ and $d=7$
fits with $p=0.541$.)




\begin{figure}
 \centering
  \includegraphics[width=\textwidth]{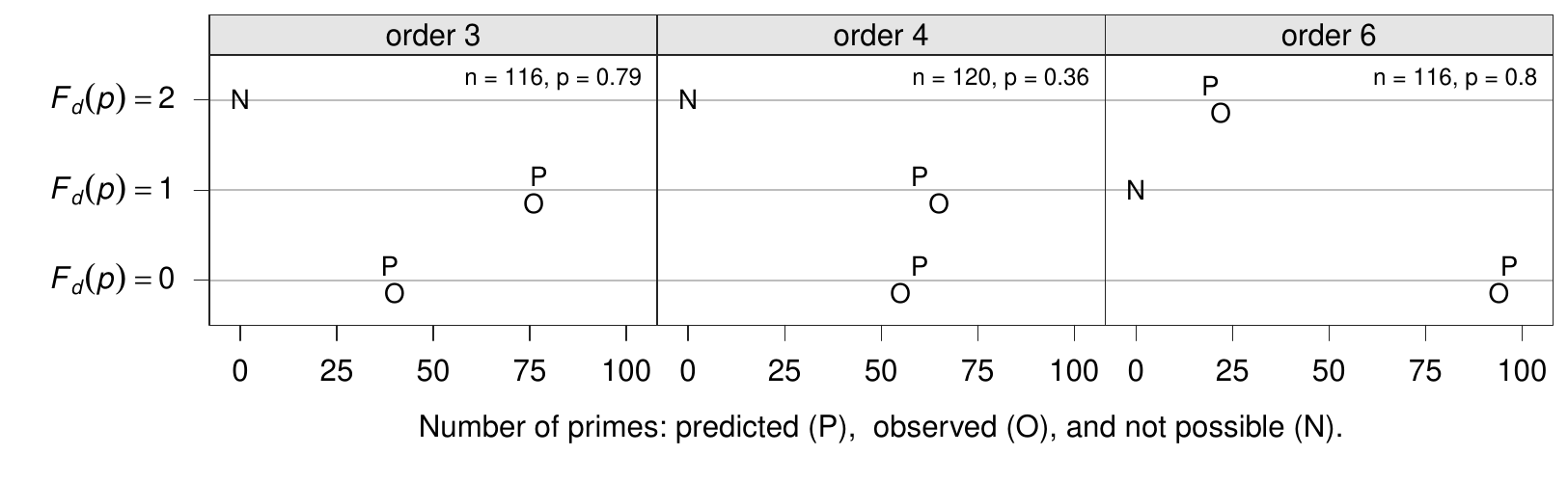}
 \caption{Predictions and observations for fixed points of order 3, 4, and 6 in
     six-digit primes} \label{fig:smallorder6}
\end{figure}

\subsection{Large Orders}

We also observed significant deviation from our predictions in the
case of large orders.  Recall that Part~\ref{Fp-1/2} of
Theorem~\ref{thm:predorders} used Proposition~2 of~\cite{REU2010} to
prove that there was at most one fixed point of order $(p-1)/2$.  In
fact, that Proposition also showed that the fixed point exists if and
only if 2 is a quadratic residue modulo $p$.  Similarly, if $3 \mid
(p-1)$ then Corollary~\ref{prop7cor} shows that there are at most two
fixed points of order $(p-1)/3$, namely $(p+2)/{3}$ and
$(2p+1)/{3}$. Using methods similar to the above we can show that
these residue classes will be fixed points when they are cubic residues modulo
$p$.

\begin{proposition} Let $p$ be a prime number equivalent to $1$ modulo
  $3$. The residue class $(p+2)/3$ is a fixed point if and only if it is a cubic
  residue modulo $p$, and similarly for  $(2p+1)/3$.
\end{proposition}
\begin{proof}
Note that since $1 \leq x \leq p-1$, Equation~\eqref{eq:spfp} is
equivalent to 
\begin{equation} \label{eq:spfpvar} 
x^{x-1} \equiv 1 \pmod{p}
\end{equation}
Then $(p+2)/3$ is a fixed point if and only
$$\left(\frac{p+2}{3}\right)^{\frac{p-1}{3}} \equiv 1 \pmod{p},$$
which by Euler's Criterion is equivalent to $(p+2)/3$ being a cubic
residue.

Similarly, if $(2p+1)/3$ is a fixed point then 
$$\left(\frac{2p+1}{3}\right)^{\frac{2p-2}{3}} \equiv 1 \pmod{p}.$$
But then 
$$\left(\frac{2p+1}{3}\right)^{\frac{p-1}{3}} \equiv
\left(\frac{2p+1}{3}\right)^{\frac{4p-4}{3}} \equiv 1 \pmod{p}$$ also,
where the first equivalence is just Fermat's Little Theorem.  So
Euler's Criterion is satisfied again.  Conversely, if 
$$\left(\frac{2p+1}{3}\right)^{\frac{p-1}{3}} \equiv 1 \pmod{p}$$
then certainly 
$$\left(\frac{2p+1}{3}\right)^{\frac{2p-2}{3}} \equiv 1 \pmod{p}$$
so $(2p+1)/3$ is a fixed point.

\end{proof}

More simplifications show that $(2p+1)/{3}
\equiv 3^{-1} \pmod{p}$ and $(p+2)/{3} \equiv 2(3^{-1}) \pmod{p}$ so
$(2p+1)/{3}$ will be a cubic residue whenever $3$ is a cubic
residue and both $(p+2)/{3}$ and $(2p+1)/{3}$ will be
cubic residues when both $2$ and $3$ are cubic residues. These same methods
can be used to show that all numbers of the form $(m({p-1})/{k})+1$
where $1 \leq m < k$ will be fixed points in the self-power map when
the number is a $k$-th residue. 

This is not quite enough to investigate $F_{(p-1)/3}(p)$ since not all
cubic residues have order equal to $(p-1)/3$.  We thus estimate the
probability that a given element of $\set{(p+2)/{3}, (2p+1)/3}$, has order
equal to exactly $(p-1)/3$.  Lemma~\ref{cz-eq7} suggests that elements
of order $d$ occur in $\mathcal{P}$ in approximately the same proportion
that they occur in the whole range $1 \leq x \leq p-1$, namely
$\phi(d)/(p-1)$.  (A more precise statement on the frequency of $p$ such
that $kd+1$ has order $d$ would appear to require some variation on
Artin's primitive root conjecture.)

We again use a binomial distribution to predict:

\begin{prediction}
\begin{enumerate}
\item $\Pr[F_{(p-1)/3}(p)=0] = \left( 1 - \dfrac{\phi((p-1)/3)}{p-1}\right)^2$
\item $\Pr[F_{(p-1)/3}(p)=1] = 2
  \left(\dfrac{\phi((p-1)/3)}{p-1}\right)\left( 1 -
    \dfrac{\phi((p-1)/3)}{p-1}\right)$ 
\item $\Pr[F_{(p-1)/3}(p)=2] = \left(\dfrac{\phi((p-1)/3)}{p-1}\right)^2$
\end{enumerate}
\end{prediction}


If $4 \mid (p-1)$, Corollary~\ref{prop7cor} shows that there
are at most three fixed points of order $(p-1)/4$, namely $(p+3)/{4}$,
$(p+1)/2$, and $(3p+1)/{4}$.  However, it turns out that they cannot
all be fixed points at the same time.

\begin{theorem}  Let $p$ be a prime number equivalent to $1$ modulo $4$.
\begin{enumerate}
\item If $p \equiv 1 \pmod{8}$, then $F_{(p-1)/4}(p) \leq 2$.
\item If $p \equiv 5 \pmod{8}$, then $F_{(p-1)/4}(p) \leq 1$.
\end{enumerate}
\end{theorem}

\begin{proof}
  Suppose $p \equiv 1 \pmod{8}$.  Since $(p+1)/2 \equiv 2^{-1}
  \pmod{p}$ and $(3p+1)/4 \equiv 4^{-1} \pmod{p}$, these two can only
  be both fixed points of order $(p-1)/4$ if $\ord_p 2 = \ord_p 4 = (p-1)/4$.
  But we know $8 \mid (p-1)$, so if $\ord_p 2 = (p-1)/4$ then $\ord_p
  4 = (p-1)/8$.  On the other hand, if $p \equiv 5 \pmod{8}$, then we
  know $\ord_p 2 \nmid (p-1)/2$ so neither $\ord_p 2$ nor $\ord_p 4$
  can be $(p-1)/4$.
\end{proof}

To make predictions on the probabilities of each number of fixed
points, we again use a binomial distribution.  If $p \equiv 1$ modulo
$8$, we keep in mind that the orders of $(p+1)/2$ and $(3p+1)/4$ are
dependent and never equal so we can treat them together:

\begin{prediction}
\begin{enumerate}
\item $\Pr[F_{(p-1)/4}(p)=0] = 
\left( 1 -
    \dfrac{\phi((p-1)/4)}{p-1}\right)
\left( 1 -
    \dfrac{3\, \phi((p-1)/4)}{(p-1)/2}\right)$
\item $\Pr[F_{(p-1)/4}(p)=1] =$ \\
$\left(\dfrac{\phi((p-1)/4)}{p-1}\right)
\left( 1 - \dfrac{3\,
      \phi((p-1)/4)}{(p-1)/2}\right)
+
 \left( 1 -
    \dfrac{\phi((p-1)/4)}{p-1}\right) 
\left(\dfrac{3\,
      \phi((p-1)/4)}{(p-1)/2}\right)$ 

\item $\Pr[F_{(p-1)/4}(p)=2] = 
\left(    \dfrac{\phi((p-1)/4)}{p-1}\right)
\left(    \dfrac{3\, \phi((p-1)/4)}{(p-1)/2}\right)$
\end{enumerate}
\end{prediction}

If 
$p \equiv 5$ modulo $8$, then we simply have:

\begin{prediction}
\begin{enumerate}
\item $\Pr[F_{(p-1)/4}(p)=0] = 
\left( 1 - \dfrac{\phi((p-1)/4)}{p-1}\right)$
\item $\Pr[F_{(p-1)/4}(p)=1] = 
\left(    \dfrac{\phi((p-1)/4)}{p-1}\right)$
\end{enumerate}
\end{prediction}

Chi-squared tests on the observed data from six-digit primes against
the distributions predicted for orders $(p-1)/3$ and $(p-1)/4$ do not
show significant deviation, as shown in Figure~\ref{fig:largeorder4}.




\begin{figure}
 \centering
  \includegraphics[width=\textwidth]{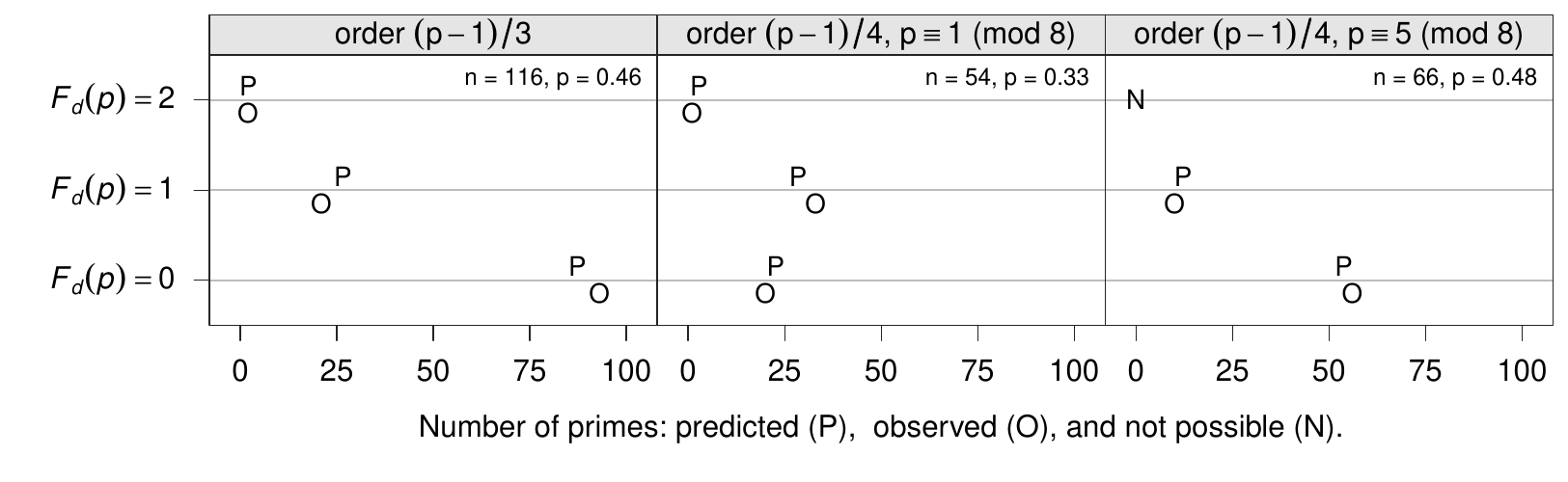}
\caption{Predictions and observations for fixed points of order
   $(p-1)/3$ and $(p-1)/4$ in six-digit primes} 
\label{fig:largeorder4}
\end{figure}

\section{Conclusion and Future Work}

In practice, it would certainly be possible for a user of the variant ElGamal
digital signature scheme to simply make sure $p$ is a safe prime, or
alternatively arrange for $r$ to always be a primitive root.  In this
way one could avoid the issue of fixed points altogether.  However, we
feel that it is very likely that a better understanding of the
self-power map will help us better understand the security of this and
other similar schemes.

We have given some bounds on the number of fixed points of the
self-power map and attempted to predict the distribution of the fixed
points using a binomial model whose mean is related to these proven
bounds.  When the order of $x$ is moderate, this binomial model
is a good predictor according to the data we collected.  When the
order of $x$ is small, in particular when it is 3, 4, or 6, the
independence assumption of the binomial model is violated in a
significant way.  However, we were able to find another model which
appears to successfully predict the distribution.

When the order of $x$ is $(p-1)/3$ or $(p-1)/4$, we once again have a
significant deviation from our first binomial model.  However, a
closer look at the set of possible fixed points in each case leads to
another binomial model which appears to be successful.  Orders in the
range $(p-1)/5$ to $(p-1)/13$ do not appear to be showing significant
deviation from the original model.  However, the sliding window
chi-squared test shows evidence of possible divergence from the
predictions in the neighborhood of $(p-1)/16$, as can be seen in
Figure~\ref{fig:chisq6-sliding} about three-quarters of the
way from the ``Order 1000'' horizontal line to the ``Order 10000''
line.  It is not clear yet whether this is a true problem with the
model, or just a ``random'' consequence of the particular primes that
we picked.  Further investigation of these orders would appear to be
the first item to be considered in future work.

Another very important item of future work would be to consider
two-cycles, namely solutions to the equations 
\begin{equation} \label{eqn:sptc}
h^h \equiv a \pmod{p} \quad \text{and} \quad a^a \equiv h \pmod{p},
\end{equation}
or more generally $k$-cycles.  Some data has been collected for these
larger cycles but the binomial distribution has not yet been
calculated or checked.  The paper~\cite{REU2010} also examined other
graph-theoretic statistics of the functional graphs created by the
self-power map, especially the number of components.  This was also
found to obey a non-normal distribution and one would like to explore
how that distribution is related to the one found here for fixed
points. 



\subsection*{Acknowledgements}
Many thanks to Richard Layton for the design and production of
Figures~\ref{fig:smallorder6} and~\ref{fig:largeorder4}.  We also
thank the Rose-Hulman Computer Science-Software Engineering and Mathematics
Departments for the use of their computers.


\begin{bibdiv}
\begin{biblist}

\bib{anghel}{thesis}{ 
  title={The Self Power Map and its Image Modulo
      a Prime},
    url={https://tspace.library.utoronto.ca/handle/1807/35765},
    type={PhD Thesis}, 
    school={University of Toronto}, 
    author={Anghel, Catalina Voichita}, 
    year={2013},
}



\bib{balog_et_al}{article}{
	title = {On the Number of Solutions of Exponential Congruences},
	author = {Balog, Antal},
        author = {Kevin A. Broughan},
        author = {Igor E. Shparlinski},
	volume = {148},
	url = {http://journals.impan.gov.pl/aa/Inf/148-1-7.html},
	doi = {10.4064/aa148-1-7},
	number = {1},
	journal = {Acta Arithmetica},
	year = {2011},
	pages = {93--103}
}




\bib{cloutier_holden}{article}{
	title = {Mapping the discrete logarithm},
	volume = {3},
	issn = {1944-4176},
	url = {http://msp.org/involve/2010/3-2/p06.xhtml},
	doi = {10.2140/involve.2010.3.197},
	number = {2},
	urldate = {2014-02-15},
	journal = {Involve, a Journal of Mathematics},
	author = {Cloutier, Daniel},
        author = {Holden, Joshua},
	year = {2010},
	pages = {197--213},
}

\bib{cobeli-zaharescu}{article}{
      author={Cobeli, Cristian},
      author={Zaharescu, Alexandru},
       title={An Exponential Congruence with Solutions in Primitive Roots},
	date={1999},
	ISSN={0035-3965},
     journal={Rev. Roumaine Math. Pures Appl.},
      volume={44},
      number={1},
       pages={15\ndash 22}
}

\bib{crocker66}{article}{
	title = {On a New Problem in Number Theory},
	volume = {73},
	number = {4},
	journal = {The American Mathematical Monthly},
	author = {Crocker, Roger},
	year = {1966},
	pages = {355--357},
}

\bib{crocker69}{article}{
	title = {On Residues of $n^n$},
	volume = {76},
	number = {9},
	journal = {The American Mathematical Monthly},
	author = {Crocker, Roger},
	year = {1969},
	pages = {1028--1029},
}


\bib{REU2010}{article}{
	title = {Structure and Statistics of the {Self-Power} Map},
	volume = {11},
	number = {2},
	journal = {{Rose-Hulman} Undergraduate Mathematics Journal},
	author = {Friedrichsen, Matthew},
        author = {Larson, Brian},
        author = {McDowell, Emily},
	year = {2010},
}

\bib{hoffman}{article}{
	title = {Statistical investigation of structure in the
          discrete logarithm}, 
	volume = {10},
	url = {http://umbracoprep.rose-hulman.edu/math/MSTR/MSTRpubs/2009/RHIT-MSTR-2009-09.pdf},
	number = {2},
	urldate = {2014-02-15},
	journal = {Rose-Hulman Undergraduate Mathematics Journal},
	author = {Hoffman, Andrew},
	year = {2009},
}

\bib{holden02}{inproceedings}{
      author={Holden, Joshua},
       title={Fixed Points and Two-Cycles of the Discrete Logarithm},
	date={2002},
   booktitle={Algorithmic number theory ({A}{N}{T}{S} 2002)},
      editor={Fieker, Claus},
      editor={Kohel, David~R.},
      series={LNCS},
   publisher={Springer},
       pages={405\ndash 415},
  url={http://link.springer-ny.com/link/service/series/0558/bibs/2369/23690405%
.htm},
}

\bib{holden02a}{misc}{
      author={Holden, Joshua},
       title={Addenda/corrigenda: Fixed Points and Two-Cycles of the Discrete
  Logarithm},
	date={2002},
        eprint = {arXiv:math/020802 [math.NT]},
	note={Unpublished, available at
          \url{http://xxx.lanl.gov/abs/math.NT/0208028}}, 
}





\bib{holden_moree}{article}{
	title = {Some Heuristics and Results for Small Cycles of the
          Discrete Logarithm}, 
	volume = {75},
	issn = {0025-5718},
	number = {253},
	journal = {Mathematics of Computation},
	author = {Joshua Holden},
        author = {Pieter Moree},
	year = {2006},
	pages = {419--449}
}

\bib{holden_robinson}{article}{
	title = {Counting Fixed Points, {Two-Cycles}, and Collisions
          of the Discrete Exponential Function using $p$-adic
          Methods}, 
	journal = {Journal of the Australian Mathematical Society},
	author = {Holden, Joshua},
        author = {Robinson, Margaret M.},
        volume = {92},
        issue = {2},
        pages = {163--178},
        year = {2012}
}

\bib{kurlberg_et_al}{article}{
    author = {Kurlberg, P\"ar},
    author = {Luca, Florian},
    author = {Shparlinski, Igor E.},
    title = {On the Fixed Points of the Map $x \mapsto x^x$ Modulo a
      Prime},
    journal = {Mathematical Research Letters}
    volume =  {22},
    issue = {1},
    year = {2015},
    pages = {141--168},
    doi = {10.4310/MRL.2015.v22.n1.a8}
}

\bib{lindle}{thesis}{
	type = {Senior Thesis},
	title = {A Statistical Look at Maps of the Discrete
          Logarithm}, 
	school = {Rose-Hulman Institute of Technology},
	author = {Lindle, Nathaniel W.},
	year = {2008}
}

\bib{handbook}{book}{
	title = {Handbook of Applied Cryptography},
	isbn = {0849385237},
	url = {http://www.cacr.math.uwaterloo.ca/hac/},
	publisher = {{CRC}},
	author = {Alfred J. Menezes},
        author = {Paul C. van Oorschot},
        author = {Scott A. Vanstone},
	year = {1996}
}

\bib{somer}{article}{
  author = {Somer, Lawrence},
  title = {The Residues of $n^n$ Modulo $p$},
  journal = {Fibonacci Quarterly},
  volume = {19},
  number = {2},
  year = {1981},
  pages = {110--117}
}  

\bib{sorenson}{article}{
	title = {Poster Abstracts from the Eighth Algorithmic Number Theory Symposium, {ANTS-8}},
	volume = {42},
	url = {http://www.sigsam.org/bulletin/articles/164/ants.pdf},
	number = {2},
	urldate = {2014-02-15},
	journal = {{ACM} Communications in Computer Algebra},
	collaborator = {Sorenson, Jonathan},
	year = {2008},
	pages = {48--66},
}

\bib{zhang}{article}{
      author={Zhang, Wen~Peng},
       title={On a Problem of {B}rizolis},
	date={1995},
	ISSN={1008-5513},
     journal={Pure Appl. Math.},
      volume={11},
      number={suppl.},
       pages={1\ndash 3}
}

\end{biblist}
\end{bibdiv}
\end{document}